%% file: beyond-choosability.tex
\documentclass[12pt]{article}
\usepackage{amsmath, amsthm, amssymb}
\usepackage{hyperref}
\usepackage[margin=1cm]{caption}
\usepackage{verbatim}
\usepackage[top=1.0in, bottom=1.0in, left=1.0in, right=1.0in]{geometry}
\usepackage{graphicx}

\usepackage{xifthen}
\usepackage{marginnote}

\usepackage{tkz-graph}
\usetikzlibrary{arrows}
\usetikzlibrary{shapes}
\usepackage[position=bottom]{subfig}

\theoremstyle{plain}
\newtheorem{thm}{Theorem}[section]
\newtheorem{prop}[thm]{Proposition}
\newtheorem{lem}[thm]{Lemma}

\newtheorem{cor}[thm]{Corollary}

\newtheorem*{mainthm}{Main Lemma}
\newtheorem*{thmA}{Theorem A}
\newtheorem*{thmB}{Theorem B}
\newtheorem*{stretching}{Stretching Lemma}
\newtheorem*{subgraph}{Subgraph Lemma}

\newtheorem*{question}{Question}

\theoremstyle{definition}

\theoremstyle{remark}

\newcommand{\fancy}[1]{\mathcal{#1}}

\newcommand{\IN}{\mathbb{N}}

\newcommand{\D}{\fancy{D}}

\newcommand{\set}[1]{\left\{ #1 \right\}}

\newcommand{\func}[3]{#1\colon #2 \rightarrow #3}

\newcommand{\irange}[1]{\left\{1,\ldots,#1\right\}}

\newcommand{\parens}[1]{\left( #1 \right)}

\newcommand{\DefinedAs}{\mathrel{\mathop:}=}

\def\D{\fancy{D}}

\def\soft#1{\leavevmode\setbox0=\hbox{h}\dimen7=\ht0\advance
	\dimen7 by-1ex\relax\if t#1\relax\rlap{\raise.6\dimen7
		\hbox{\kern.3ex\char'47}}#1\relax\else\if T#1\relax
	\rlap{\raise.5\dimen7\hbox{\kern1.3ex\char'47}}#1\relax
	\else\if d#1\relax\rlap{\raise.5\dimen7\hbox{\kern.9ex
			\char'47}}#1\relax\else\if D#1\relax\rlap{\raise.5\dimen7
		\hbox{\kern1.4ex\char'47}}#1\relax\else\if l#1\relax
	\rlap{\raise.5\dimen7\hbox{\kern.4ex\char'47}}#1\relax
	\else\if L#1\relax\rlap{\raise.5\dimen7\hbox{\kern.7ex
			\char'47}}#1\relax\else\message{accent \string\soft
		\space #1 not defined!}#1\relax\fi\fi\fi\fi\fi\fi} 

\def\erdos{Erd\H{o}s}
\def\hladky{Hladk{\'y}}
\def\kral{Kr{\'a}{\soft{l}}}

\def\adj{\leftrightarrow}
\def\nonadj{\not\!\leftrightarrow}

\renewcommand{\restriction}{\mathord{\upharpoonright}}
\renewcommand{\vec}{\overrightarrow}

\newcommand{\aside}[1]{\marginnote{\scriptsize{#1}}[0cm]}
\newcommand{\aaside}[2]{\marginnote{\scriptsize{#1}}[#2]}

\begin{document}
	\title{Beyond Degree Choosability}
	\author{Daniel W. Cranston\thanks{Department of Mathematics and Applied
			Mathematics, Viriginia Commonwealth University, Richmond, VA;
			\texttt{dcranston@vcu.edu}; 
			Research of the first author is partially supported by NSA Grant
			98230-15-1-0013.}
		\and
		Landon Rabern\thanks{LBD Data Solutions, Lancaster, PA;
			\texttt{landon.rabern@gmail.com}}
	}
	
	\maketitle
	\begin{abstract}
		Let $G$ be a connected graph with maximum degree $\Delta$.  Brooks' theorem
		states that $G$ has a $\Delta$-coloring unless $G$ is a complete graph or an
		odd cycle.  A graph $G$ is \emph{degree-choosable} if $G$ can be properly
		colored from its lists whenever each vertex $v$ gets a list of $d(v)$ colors.
		In the context of list coloring, Brooks' theorem can be strengthened to the
		following.  Every connected graph $G$ is degree-choosable unless each block of
		$G$ is a complete graph or an odd cycle; such a graph $G$ is a \emph{Gallai
			tree}.  
		
		This degree-choosability result was further strengthened to Alon--Tarsi orientations;
		these are orientations of $G$ in which the number of spanning Eulerian
		subgraphs with an even number of edges differs from the number with an odd
		number of edges.  A graph $G$ is \emph{degree-AT} if $G$ has an Alon--Tarsi
		orientation in which each vertex has indegree at least 1.  
		Alon and Tarsi showed that if $G$ is degree-AT, then $G$ is also
		degree-choosable.
		\hladky, \kral, and
		Schauz showed that a connected graph is degree-AT if and only if it is not a
		Gallai tree.  In this paper, we consider pairs $(G,x)$ where $G$ is a connected
		graph and $x$ is some specified vertex in $V(G)$.  We characterize pairs such
		that $G$ has no Alon--Tarsi orientation in which each vertex has indegree at
		least 1 and $x$ has indegree at least 2.  When $G$ is 2-connected, the
		characterization is simple to state.
	\end{abstract}
	
	\section{Introduction}
	
	Brooks' theorem is one of the fundamental results in graph coloring.
	For every connected graph $G$, it says that $G$ has a $\Delta$-coloring
	unless $G$ is a complete graph $K_{\Delta+1}$ or an odd cycle.  When we seek to
	prove coloring results by induction, we often want to color a subgraph $H$
	where different vertices have different lists of allowable colors (those not
	already used on their neighbors in the coloring of $G-H$).  This gives rise to
	list coloring.  Vizing~\cite{vizing1976} and, independently, \erdos,
	Rubin, and Taylor~\cite{ERT} extended Brooks' theorem to list coloring. They
	proved an analogue of Brooks' theorem when each vertex $v$ has $\Delta$
	allowable colors (possibly different colors for different vertices).  
	\erdos, Rubin, and Taylor~\cite{ERT} and Borodin~\cite{borodin1977criterion} 
	strengthened this Brooks' analogue to the following result, where a
	\emph{Gallai tree}\aside{Gallai tree} is a connected graph in which each
	block is a complete graph
	or an odd cycle.  
	
	\begin{thmA}
		If $G$ is connected and not a Gallai tree, then for any list assignment $L$
		with $|L(v)|=d(v)$ for all $v\in V(G)$, graph $G$ has a proper coloring
		$\varphi$ with $\varphi(v)\in L(v)$ for all $v$.
	\end{thmA}
	
	The graphs in Theorem~A are
	\emph{degree-choosable}\aaside{degree-choosable}{-.3cm}. 
	It is easy to check that every Gallai tree is not degree-choosable.  So the set
	of all connected graphs that are not degree-choosable are precisely the Gallai trees.
	\hladky, \kral, and Schauz~\cite{HKS} extended this characterization to the setting of
	Alon--Tarsi orientations.
	
	For any digraph $D$, a \emph{spanning Eulerian subgraph} is one in which each vertex
	has indegree equal to outdegree.  The \emph{parity} of a spanning Eulerian
	subgraph is the parity of its number of edges.  For an orientation of a graph
	$G$, let EE (resp.~EO) denote the number of even (resp.~odd) spanning Eulerian
	subgraphs.  An orientation is \emph{Alon--Tarsi}\aaside{Alon--Tarsi orientation}{-.4cm}
	(or AT) if EE and EO differ. 
	A graph $G$ is \emph{$f$-AT}\aside{$f$-AT, $k$-AT} if it has an Alon--Tarsi
	orientation $D$ such that $d^+(v)\le f(v)-1$ for each vertex $v$.  In
	particular, $G$ is \emph{degree-AT}\aside{degree-AT} (resp.~\emph{$k$-AT}) if
	it is $f$-AT, where $f(v)=d(v)$ (resp.~$f(v)=k$) for all $v$.  Similarly, a
	graph $G$ is \emph{$f$-choosable}\aside{$f$-choosable} if $G$ has a proper
	coloring $\varphi$ from any list assignment $L$ such that $|L(v)|=f(v)$ for all
	$v\in V(G)$.  Alon and Tarsi~\cite{AlonT} used algebraic methods to prove the
	following theorem for choosability.  Later, Schauz~\cite{schauz2010flexible}
	strengthened the result to paintability, which we discuss briefly in
	Section~\ref{extensions}.  
	
	\begin{thmB}
				\hypertarget{target:thmB}{}
		For a graph $G$ and $\func{f}{V(G)}{\IN}$,
		if $G$ is $f$-AT, then $G$ is also $f$-choosable.
	\end{thmB}
	
	In this paper we characterize those graphs $G$ with a specified vertex $x$
	that are not $f$-AT, where $f(x)=d(x)-1$ and $f(v)=d(v)$ for all other $v\in
	V(G)$.  All such graphs are formed from a few 2-connected building blocks, by
	repeatedly applying a small number of operations.  
	Most of the work in the proof is spent on the case when $G$ is 2-connected.
	This result is easy to state, so we include it a bit later in the introduction,
	as our \hyperlink{target:mainLemma}{Main Lemma}.
	Near the end of Section~\ref{MainThmSec}, with a little more work we extend our
	\hyperlink{target:mainLemma}{Main Lemma}, by removing the hypothesis of
	2-connectedness, to characterize
	all pairs $(G,h_x)$ that are not AT.  This result is
	Theorem~\ref{thm:1connected}.
	
	This line of research began with Gallai, who studied the minimum number of edges
	in an $n$-vertex $k$-critical graph $G$.  Since $G$ has minimum degree at least
	$k-1$, clearly $|E(G)|\ge\frac{k-1}2n$.  Gallai~\cite{gallai1963kritische} improved this
	bound by classifying all connected subgraphs that can be induced by vertices of
	degree $k-1$ in a $k$-critical graph.  By Theorem~A, all such graphs are Gallai
	trees.  Here, we consider graphs $G$ that are critical with respect to Alon--Tarsi
	orientation.  Specifically, $G$ is not $(k-1)$-AT, but every proper subgraph is;
	such graphs are \emph{$k$-AT-critical}.  The characterization of degree-AT
	graphs shows that, much like $k$-critical graphs, in a $k$-AT-critical graph
	$G$, every connected subgraph induced by vertices of degree $k-1$ must be a
	Gallai tree.  Our main result characterizes the subgraphs that can be induced by
	vertices of degree $k-1$, together with a single vertex of degree $k$.  Thus, it
	is natural to expect that this result will lead to improved lower bounds on
	the number of edges in $n$-vertex $k$-AT-critical graphs.
	
	Similar to that for
	degree-AT, our characterization remains unchanged in the contexts of
	list-coloring and paintability, as we show in Section~\ref{extensions}.  We
	see a sharp contrast when we consider
	graphs $G$ with two specified vertices $x_1$ and $x_2$ that are not $f$-AT,
	where $f(x_i)=d(x_i)-1$ for each $i\in \{1,2\}$ and $f(v)=d(v)$ for all other
	$v\in V(G)$.  For Alon--Tarsi orientations, we have more than 50 exceptional
	graphs on seven vertices or fewer.  Furthermore, the characterizations for
	list-coloring, paintability, and Alon--Tarsi orientations all differ.
	
	\bigskip
	
	We consider graphs with vertices labeled by natural numbers; that is, pairs
	$(G,h)$ where $G$ is a graph and $\func{h}{V(G)}{\IN}$.  We focus on the case
	when $h(x)=1$ for some $x$ and $h(v)=0$ for all other $v$; we denote this
	labeling as $h_x$. \aaside{$h_x$}{-.45cm}  We say that \emph{$(G, h)$ is
		AT}\aside{$(G,h)$\\ is AT} if $G$ is $(d_G - h)$-AT.
	When $H$ is an induced subgraph of $G$, we simplify notation by referring to
	the pair $(H, h)$ when we really mean $\parens{H, h\restriction_{V(H)}}$.
	
	Given a pair $(G,h)$ and a specified edge $e\in E(G)$, when we \emph{stretch
		$e$}\aside{stretch $e$}, we form $(G',h')$ from $(G,h)$ by subdividing $e$ twice and setting
	$h'(v_i)=0$ for each of the two new vertices, $v_1$ and $v_2$ (and $h'(v)=h(v)$
	for all other vertices $v$).  In Section~\ref{prelims}, we prove a
	\hyperlink{target:SubdivideTwice}{Stretching Lemma}, which shows that if $(G,h)$ is not
	AT and $e\in E(G)$, then stretching $e$ often yields another pair $(G',h')$
	that is also not AT.  Thus, stretching plays a key role in our main result.
	
	It is easy to check that the three pairs $(G,h)$ shown in
Figure~\ref{fig:seeds} are not AT
	(and we do this below, in Proposition~\ref{prop:easyD}).  Let $\D$ \aside{$\D$} be the
	collection of all pairs formed from the graphs in Figure~\ref{fig:seeds} by stretching each
	bold edge 0 or more times.  The \hyperlink{target:SubdivideTwice}{Stretching Lemma}
	implies that each pair in $\D$ is not AT.  Our
	\hyperlink{target:mainLemma}{Main Lemma} is that these are
	the only pairs $(G,h_x)$, where $G$ is 2-connected and neither complete 
	nor an odd cycle, such that $(G,h_x)$ is not AT, for some vertex $x\in V(G)$.  
	
	\begin{mainthm}
		\hypertarget{target:mainLemma}{}
		Let $G$ be 2-connected and let $x \in V(G)$.
		Now $(G,h_x)$ is AT if and only if
		\begin{enumerate}
			\item[(1)] $d(x)=2$ and $G-x$ is not a Gallai tree; or
			\item[(2)] $d(x)\ge 3$, $G$ is not complete, and $(G,h_x) \not \in \D$.
		\end{enumerate}
	\end{mainthm}
	
	The characterization of degree-choosable graphs has been applied
	to prove a variety of graph coloring
	results~\cite{BohmeMS, CranstonPTV, KostochkaS, KralS, Thomassen-surface}. 
	Likewise, we think our main results in this paper may be helpful in proving
	other results for Alon--Tarsi orientations, such as giving better lower bounds
	on the number of edges in $k$-AT-critical graphs.
	\bigskip
	
	\input{D-pics.tex}

	To conclude this section, we show that each pair in $\D$ is not AT.
	\begin{prop}
		If $(G,h_x)\in \D$, then $(G,h_x)$ is not AT.
		\label{prop:easyD}
	\end{prop}
	\begin{proof}
		For each pair $(G,h_x)\in \D$, we construct a list assignment $L$ such that
		$|L(x)|=d(x)-1$ and $|L(v)|=d(v)$ for all other $v\in V(G)$, but $G$ has no
		proper coloring from $L$.  Now $(G,h_x)$ is not AT, by the contrapositive of
		Theorem~B.
		
		Let $(G,h_x)$ be some stretching of the leftmost pair in
Figure~\ref{fig:seeds}. Assign the
		list $\{1,2,3\}$ to each of the vertices on the unbolded triangle and assign the
		list $\{1,2\}$ to each other vertex.  If $G$ has some coloring from these lists,
		then vertex $x$, labeled 1 in the figure, must get color 1 or 2; by symmetry,
		assume it is 1.  Along each path from $x$ to the triangle, colors must
		alternate $2, 1, \ldots$.  Each of the paths from $x$ to the triangle has odd
		length; thus, color 1 is forbidden from appearing on the triangle.  So $G$ has
		no coloring from $L$.  Now let $(G,h_x)$ be some stretching of the center pair
		in Figure~\ref{fig:seeds}.  The proof is identical to the first case, except that each path
		has even length, so if $x$ gets color 1, then color 2 is forbidden on the triangle.
		
		Finally, consider the rightmost pair in Figure~\ref{fig:seeds}.  Here $d(x)=4$ and $d(v)=3$
		for all other $v\in V(G)$.  Thus, it suffices to show that $G$ is not
		3-colorable.  Assume that $G$ has a 3-coloring and, by symmetry, assume that $x$
		is colored 1.  Now colors 2 and 3 must each appear on two neighbors of $x$.
		Thus, the two remaining vertices must be colored 1.  Since they are adjacent,
		this is a contradiction, which proves that $G$ is not 3-colorable.
	\end{proof}
	
	\section{Subgraphs, subdivisions, and cuts}
	\label{prelims}
	
	When \hladky, \kral, and Schauz characterized degree-AT graphs, their proof relied
	heavily on the observation that a connected graph $G$ is degree-AT if and only
	if $G$ has some induced subgraph $H$ such that $H$ is degree-AT.  Below, we
	reprove this easy lemma, and also extend it to our setting of pairs $(G,h_x)$.

	\begin{subgraph}
		\hypertarget{target:InducedSubgraph}{}
		Let $G$ be a connected graph and let $H$ be an induced subgraph of $G$.  If $H$
		is degree-AT, then also $G$ is degree-AT.  Similarly, if $x\in V(H)$ and
		$(H,h_x)$ is AT, then also $(G,h_x)$ is AT.  Further, if $x\notin V(H)$,
		$d_G(x)\ge 2$, and $(H,h_x)$ is AT, then $(G,h_x)$ is AT.  
	\end{subgraph}
	\begin{proof}
		Suppose that $H$ is degree-AT, and let $D'$ be an orientation of $H$ showing
		this.  Extend $D'$ to an orientation $D$ of $G$ by orienting all edges away from
		$H$, breaking ties arbitrarily, but consistently.  Now every directed cycle in
		$D$ is also a directed cycle in $D'$ (and vice versa), so $G$ is degree-AT.
		The proof of the second statement is identical.  The proof of the third
		statement is similar, but now if some edge $xy$ has endpoints equidistant from
		$H$, then $xy$ should be oriented into $x$.
	\end{proof}
	
	Recall that, given a pair $(G,h)$ and a specified edge $e\in E(G)$, when we
	\emph{stretch $e$}, we form $(G',h')$ from $(G,h)$ by subdividing $e$ twice
	and setting $h'(v_i)=0$ for each of the two new vertices, $v_1$ and $v_2$
	(and $h'(v)=h(v)$ for all other vertices $v$).  By repeatedly stretching edges,
	starting from the three pairs in Figure~\ref{fig:seeds}, we form all pairs
	$(G,h_x)$, where $G$ is 2-connected and $(G,h_x)$ is not AT.  
	The following lemma will be useful for proving this.
	
	\begin{stretching}\hypertarget{target:SubdivideTwice}{}
		Form $(G',h')$ from $(G,h)$ by stretching some edge $e\in E(G)$.
		Now
		\begin{enumerate}
			\item[(1)] if $(G,h)$ is AT, then $(G', h')$ is AT; and
			\item[(2)] if $(G', h')$ is AT, then either $(G,h)$ is AT or $(G-e,h)$ is AT.
		\end{enumerate}	
	\end{stretching}
	\begin{proof}
		Suppose $e = u_1u_2$ and call the new vertices $v_1$ and $v_2$ so that $G'$ contains
		the induced path $u_1v_1v_2u_2$.  For (1), let $D$ be an orientation of $G$ showing
		that $(G,h)$ is AT. By symmetry we may assume $u_1u_2 \in E(D)$. Form an
		orientation $D'$ of $G'$ from $D$ by replacing $u_1u_2$ with the directed path
		$u_1v_1v_2u_2$.  We have a natural parity preserving bijection between the spanning
		Eulerian subgraphs of $D$ and $D'$, so we conclude that $(G', h')$ is AT.
		
		For (2), let $D'$ be an orientation of $G'$ showing that $(G',h')$ is AT. 
		Suppose $G'$ contains the directed path $u_1v_1v_2u_2$ or the directed path
		$u_2v_2v_1u_1$.  By symmetry, we can assume it is $u_1v_1v_2u_2$.  Now form an
		orientation $D$ of $G$ by replacing $u_1v_1v_2u_2$ with the directed edge
		$u_1u_2$.  As above, we have a parity preserving bijection between the spanning
		Eulerian subgraphs of $D$ and $D'$, so we conclude that $(G, h)$ is AT.  So
		suppose instead that $G'$ contains neither of the directed paths $u_1v_1v_2u_2$
		and $u_2v_2v_1u_1$.  Now no spanning Eulerian subgraph of $D'$ contains a cycle
		passing through $v_1$ and $v_2$.  So, the spanning Eulerian subgraph counts of
		$D'$ are the same as those of $D' - v_1 - v_2$.  However, this gives an
		orientation of $G-e$ showing that $(G-e, h)$ is AT.  
	\end{proof}
	
	Given a pair $(G,h)$ that is not AT, the \hyperlink{target:SubdivideTwice}{Stretching Lemma} suggests a way to construct a larger graph $G'$ such that $(G',h')$ is
	not AT.  In some cases, we can also use the
	\hyperlink{target:SubdivideTwice}{Stretching Lemma} to construct a smaller graph
	$\widehat{G}$ such that $(\widehat{G},h)$ is not AT.  
	Specifically, we have the following.
	
	\begin{cor}
		\label{SubdivideConstructor}
		\label{ReduceP4Cor}
		If $e$ is an edge in $G$ such that $(G,h)$ is not AT and $(G-e, h)$ is not AT,
		then stretching $e$ gives a pair $(G',h')$ that is not AT.  Further,
		let $G$ be a graph with an induced path $u_1v_1v_2u_2$ such that $d_G(v_1) =
		d_G(v_2) = 2$.  If $(G,h)$ is AT, where $h(v_1) = h(v_2) = 0$, and
		$(G-v_1-v_2,h)$ is not AT, then \[\parens{(G - v_1 - v_2) + u_1u_2,
			h\restriction_{V(G) \setminus \set{v_1, v_2}}} \text{ is AT.}\]
	\end{cor}
	\begin{proof}
		The first statement is immediate from the \hyperlink{target:SubdivideTwice}{Stretching Lemma}.  Now we prove the second.  Suppose $(G,h)$ satisfies the
		hypotheses.  Applying part (2) of the \hyperlink{target:SubdivideTwice}{Stretching Lemma}
		shows that either $\parens{G - v_1 - v_2, h\restriction_{V(G) \setminus
				\set{v_1, v_2}}}$ is AT or $\parens{(G - v_1 - v_2) + u_1u_2,
			h\restriction_{V(G) \setminus \set{v_1, v_2}}}$ is AT. 
		By hypothesis, the former is false.  Thus, the latter is true.
	\end{proof}
	
	With standard vertex coloring, we can easily reduce to the case where $G$ is
	2-connected.  If $G$ is a connected graph with two blocks, $B_1$ and $B_2$,
	meeting at a cutvertex $x$, then we can color each of $B_1$ and $B_2$
	independently, and afterward we can permute colorings to match at $x$.
	For Alon--Tarsi orientations, the situation is not quite as simple.  However,
	the following lemma plays a similar role for us.
	
	\begin{lem}\label{CutvertexPatch}
		Let $A_1, A_2 \subseteq V(G)$, and $x\in V(G)$ be such that $A_1\cup A_2=V(G)$ and $A_1 \cap
		A_2 = \set{x}$.  If $G[A_i]$ is $f_i$-AT for each $i \in \{1,2\}$, then $G$ is
		$f$-AT, where $f(v) = f_i(v)$ for each $v \in V(A_i-x)$ and $f(x) = f_1(x) + f_2(x)
		- 1$.  Going the other direction, if $G$ is $f$-AT, then $G[A_i]$ is $f_i$-AT
		for each $i \in \{1,2\}$, where $f_i(v) = f(v)$ for each $v \in V(A_i-x)$ and
		$f_1(x) + f_2(x) \le f(x) + 1$.
	\end{lem}
	\begin{proof}
		We begin with the first statement.
		For each $i \in \{1,2\}$, choose an orientation $D_i$ of $A_i$ showing that $A_i$
		is $f_i$-AT.  Together these $D_i$ give an orientation $D$ of $G$. Since no cycle
		has vertices in both $A_1-x$ and $A_2-x$, we have
		\begin{align*}
			EE(D) - EO(D) &= EE(D_1)EE(D_2) + EO(D_1)EO(D_2) - EE(D_1)EO(D_2) - EO(D_1)EE(D_2) \\
			&= (EE(D_1) - EO(D_1))(EE(D_2) - EO(D_2)) \\
			&\ne 0.
		\end{align*}
		Hence $G$ is $f$-AT.
		
		Now we prove the second statement.  Suppose that $G$ is $f$-AT and choose an
		orientation $D$ of $G$ showing this. 
		Let $D_i = D[A_i]$ for each $i \in \{1,2\}$.  As above, we have $0 \ne
		EE(D) - EO(D) = (EE(D_1) - EO(D_1))(EE(D_2) - EO(D_2))$. Hence, $EE(D_1) -
		EO(D_1) \ne 0$ and $EE(D_2) - EO(D_2) \ne 0$.  Since the indegree of $x$ in
		$D$ is the sum of the indegree of $x$ in $D_1$ and the indegree of $x$ in
		$D_2$, the lemma follows.
	\end{proof}
	
	\section{Degree-AT graphs and an Extension Lemma}
	Recall that our \hyperlink{target:mainLemma}{Main Lemma} relies on a
	characterization of degree-AT graphs. 
	As we mentioned in the introduction, a description of degree-choosable graphs
	was first given by Borodin~\cite{borodin1977criterion} and \erdos, Rubin, and
	Taylor~\cite{ERT}.  \hladky, \kral, and Schauz~\cite{HKS} later extended the
	proof from~\cite{ERT} to Alon--Tarsi orientations.  This proof relies on
	Rubin's Block lemma, which states that every 2-connected graph $G$ contains an
	induced even cycle with at most one chord, unless $G$ is complete or an
	odd cycle.  For variety, and completeness, we include a new proof; it extends
	ideas of Kostochka, Stiebitz, and Wirth~\cite{KSW} from list-coloring to
	Alon--Tarsi orientations.  For this proof we need the following very special
	case of a key lemma in \cite{OreVizing}.  When vertices $x$ and $y$ are
	adjacent, we write $x\adj y$; otherwise $x\nonadj y$.
	
	\begin{lem}\label{GeneralEulerLemma}
		Let $G$ be a graph and $x \in V(G)$ such that $H$ is connected, where $H \DefinedAs G-x$. 
		If there exist $z_1, z_2 \in V(H)$ with $N_H[z_1] = N_H[z_2]$ such that $x \adj
		z_1$ and $x \nonadj z_2$, then $G$ is $f$-AT where $f(x) = 2$ and $f(v) =
		d_G(v)$ for all $v \in V(H)$.
	\end{lem}
	\begin{proof}
		Order the vertices of $H$ with $z_1$ first and $z_2$ second so that every
		vertex, other than $z_1$, has at least one neighbor preceding it. 
		Orient each edge of $H$ from its earlier endpoint toward its later endpoint. 
		Orient $xz_1$ into $z_1$ and orient all other
		edges incident to $x$ into $x$.  Let $D$ be the resulting orientation. 
		Clearly, $d_{D}^+(v) \le f(v) - 1$ for all $v \in V(D)$.  So, we just need to
		check that $EE(D) \ne EO(D)$.  
		
		Since $xz_1$ is the only edge of $D$ leaving
		$x$, and $D-x$ is acyclic, every spanning Eulerian subgraph of $D$ that has
		edges must have edge $xz_1$.  
		Consider an Eulerian subgraph $A$ of $D$ containing $xz_1$. Since $z_1$ 
		has indegree $1$ in $A$, it must also have outdegree $1$ in $A$.  We show
		that $A$ has a mate $A'$ of opposite parity.  
		If $z_2 \in A$ then $z_1z_2w \in A$, for some $w$, so we form
		$A'$ from $A$ by removing $z_1z_2w$ and adding $z_1w$. 
		If instead $z_1z_2\notin A$, then $z_2 \not \in A$ and $z_1w \in
		A$ for some $w \in N_H[z_1]-z_2$, so we form $A'$ from $A$ by removing $z_1w$ and
		adding $z_1z_2w$.  
		Hence exactly half of the Eulerian subgraphs of $D$ that contain edges are
		even.  Since the edgeless spanning subgraph of $D$ is an even Eulerian
		subgraph, we conclude that $EE(D) = EO(D) + 1$.  Hence $G$ is $f$-AT.
	\end{proof}
	
	We use the previous lemma to give a new proof of the characterization of
	degree-AT graphs.
	\begin{lem}
		\label{DegreeATClassification}
		A connected graph $G$ is degree-AT if and only if it is not a Gallai tree.
	\end{lem}
	\begin{proof}
		We begin with the ``only if'' direction.  Neither odd cycles nor complete graphs
		are degree-choosable.  Thus, by \hyperlink{target:thmB}{Theorem B}, they are not
		degree-AT.  By induction on the number of blocks, Lemma~\ref{CutvertexPatch}
		implies that no Gallai tree is degree-AT.
		
		Now, the ``if'' direction.
		Suppose there exists a connected graph that is not a Gallai tree, but is also not
		degree-AT.  Let $G$ be such a graph with as few vertices as possible.
		Since $G$ is not degree-AT, no induced subgraph $H$ of $G$ is
		degree-AT by the \hyperlink{target:InducedSubgraph}{Subgraph Lemma}. 
		Hence, for any $v \in V(G)$ that is not a cutvertex, $G-v$ must be a Gallai
		tree by minimality of $|G|$.  
		
		If $G$ has more than one block, then for endblocks $B_1$ and $B_2$, choose
		noncutvertices $w\in B_1$ and $x\in B_2$.  By the minimality of $|G|$, both
		$G-w$ and $G-x$ are Gallai trees.  Since every block of $G$ appears either as a
		block of $G-w$ or as a block of $G-x$, every block of $G$ is either complete or
		an odd cycle.  Hence, $G$ is a Gallai tree, a contradiction.  So instead $G$
		has only one block, that is, $G$ is $2$-connected.  Further, $G-v$ is a Gallai
		tree for all $v \in V(G)$.
		
		Let $v$ be a vertex of minimum degree in $G$.  Since $G$ is $2$-connected,
		$d_G(v) \ge 2$ and $v$ is adjacent to a noncutvertex in every endblock of $G-v$.
		If $G-v$ has a complete block $B$ with noncutvertices $x_1,x_2$ where $v \adj
		x_1$ and $v \nonadj x_2$, then we can apply Lemma \ref{GeneralEulerLemma} 
		to conclude that $G$ is degree-AT, a
		contradiction.  So, $v$ must be adjacent to every noncutvertex in every
		complete endblock of $G-v$.
		
		Suppose $d_G(v) \ge 3$.  Now no endblock of $G-v$ can be an odd cycle of
		length at least $5$ ($G$ would have vertices of degree $3$ and also $d_G(v) \ge
		4$, contradicting the minimality of $d_G(v)$).  Let $B$ be a smallest complete
		endblock of $G-v$.  Now for a noncutvertex $x \in V(B)$, we have $d_G(x) =
		|B|$ and hence $d_G(v) \le |B|$. 
		If $G-v$ has at least two endblocks, then $2(|B|-1) \le |B|$, so $d_G(v)
		\le |B| = 2$, a contradiction.  Hence, $G-v = B$ and $v$ is joined to $B$, so
		$G$ is complete, which is a contradiction.
		
		Thus, we have $d_G(v) = 2$.  Suppose $G-v$ has at least two endblocks. 
		Now it has exactly two and $v$ is adjacent to one noncutvertex in each. 
		Neither of the endblocks can be odd cycles of length at least five, since then
		we can get a smaller counterexample by the \hyperlink{target:SubdivideTwice}{Stretching Lemma}.  Since $v$ is adjacent to every noncutvertex in every
		complete endblock of $G-v$, both endblocks must be $K_2$.  But now either
		$G=C_4$ (which is degree-AT, by orienting the cycle consistently) or we can get
		a smaller counterexample by the \hyperlink{target:SubdivideTwice}{Stretching Lemma}. 
		So, $G-v$ must be $2$-connected. Since $G-v$ is a Gallai tree, it is either
		complete or an odd cycle.  If $G-v$ is not complete, then we can get a smaller
		counterexample by the \hyperlink{target:SubdivideTwice}{Stretching Lemma}.  So, $G-v$
		is complete and $v$ is adjacent to every noncutvertex of $G-v$; that is, $G$ is
		complete, a contradiction.
	\end{proof}
	
	\section{When h is 1 for one vertex}
	\label{MainThmSec}
	
	In this section, we prove our \hyperlink{target:mainLemma}{Main Lemma}.
	For a graph $G$ and $x \in V(G)$ recall that
	$\func{h_x}{V(G)}{\IN}$ is defined as $h_x(x) = 1$ and $h_x(v) = 0$ for all $v
	\in V(G-x)$. We classify the connected graphs $G$ such that $(G,h_x)$ is AT for
	some $x \in V(G)$.  We begin with the case when $G$ is 2-connected, which takes
	most of the work.  At the end of the section, we extend our characterization to
	all connected graphs.
	
	We will show that for most 2-connected graphs $G$ and vertices $x\in V(G)$, the
	pair $(G,h_x)$ is AT.  Specifically, this is true for all pairs except those in
	$\D$, defined in the introduction.  In view of the
	\hyperlink{target:InducedSubgraph}{Subgraph Lemma}, for a 2-connected graph $G$ and
	$x\in V(G)$, to show that $(G,h_x)$ is AT it suffices to
	find some induced subgraph $H$ such that $(H,h_x)$ is AT. 
	The subgraphs $H$ that we consider all have $d_H(x)=0$ or $d_H(x)\ge 3$.  This
	motivates the next lemma, which allows us to reduce to the case $d_G(x)\ge 3$.
	
	\begin{lem}
		\label{DegreeTwoVertex}
		If $G$ is a connected graph and $x \in V(G)$ with $d_G(x) = 2$, then $(G,h_x)$
		is AT if and only if $G-x$ is degree-AT.
	\end{lem}
	\begin{proof}
		Let $D$ be an orientation of $G$ showing that $(G,h_x)$ is AT.  Now
		$d_{D}^-(x) = 2$, so no spanning Eulerian subgraph contains a cycle
		passing through $x$.  Therefore, the Eulerian subgraph counts in $G-x$ are
		different and $G-x$ is degree-AT.  The other direction is immediate from the
		\hyperlink{target:InducedSubgraph}{Subgraph Lemma}.
	\end{proof}
	
	Lemma~\ref{DegreeTwoVertex}, together with Lemma~\ref{DegreeATClassification},
	proves Case (1) of our \hyperlink{target:mainLemma}{Main Lemma}.  Before we can prove Case
	(2), we need a few more definitions and lemmas.
	A \emph{$\theta$-graph}\aside{$\theta$-graph} consists of two vertices joined by
	three internally disjoint paths, $P_1$, $P_2$, and $P_3$.  When we write $h_x$
	for a $\theta$-graph, we always assume that $d(x)=3$.  We will see
	shortly that if $H$ is a $\theta$-graph with $d_H(x)=3$, then $(H,h_x)$ is AT.
	Thus, the \hyperlink{target:InducedSubgraph}{Subgraph Lemma} implies that 
	if $(G,h_x)$ is not AT, then $G$ has no induced $\theta$-graph $H$ with $d_H(x)=3$. 
	A \emph{$T$-graph}\aside{$T$-graph} is formed from vertices $x, z_1, z_2, z_3$,
	by making the $z_i$ 
	pairwise adjacent, and joining each vertex $z_i$ to $x$ by a path $P_i$ (where
	the $P_i$ are disjoint).  Equivalently, a $T$-graph is formed from $K_4$ by
	subdividing each of the edges incident to $x$ zero or more times.
	
	Similar to the proof characterizing degree-AT graphs in~\cite{HKS},
	our approach in proving our \hyperlink{target:mainLemma}{Main Lemma}
	is to find an induced subgraph $H$ such that $(H,h_x)$ is AT, and apply the
	\hyperlink{target:InducedSubgraph}{Subgraph Lemma}.
	Thus, we need the following lemma about pairs $(H,h_x)$ that are AT.
	
	\begin{lem}
		\label{ThetaReducible}
		\label{TgraphReducible}
		\label{T+graphReducible}
		The pair $(G,h_x)$ is AT whenever (i) $G$ is a $\theta$-graph, (ii) $G$ is a
		$T$-graph and two paths $P_i$ have lengths of opposite parities, or (iii) $G$ is
		formed from a $T$-graph by adding an extra vertex with neighborhood
		$\{z_1,z_2,z_3\}$.
	\end{lem}
	\begin{proof}
		In each case, we give an AT orientation $D$ of $G$ such that $d_D^-(v)\ge
		h_x(v)+1$ for each $v\in V(G)$.
		
		Case (i).  Orient the edges of each path $P_i$ consistently, with $P_1$ and
		$P_2$ into $x$ and $P_3$ out of $x$; this orientation satisfies the degree
		requirements.  Further, it has exactly three spanning Eulerian subgraphs,
		including the empty subgraph.  Thus, $EE+EO$ is odd, so $EE\ne EO$.
		
		Case (ii).  Let $P_1$ and $P_2$ be two paths with opposite parities.  As before,
		orient the edges of each path consistently, with $P_1$ and $P_2$ into $x$ and
		$P_3$ out of $x$.  Orient the three additional edges as $\vec{z_1z_2},
		\vec{z_2z_3}$, and $\vec{z_3z_1}$.  The resulting digraph $D$ has four spanning
		Eulerian subgraphs, 3 of one parity and 1 of the other.  Note that the empty
		subgraph and the subgraph $\{\vec{z_1z_2}, \vec{z_2z_3}, \vec{z_3z_1}\}$ have
		opposite parities.  Further, the parities are the same for the two subgraphs
		consisting of the directed cycles $xP_3z_3z_1P_1$ and $xP_3z_3z_1z_2P_2$.  So,
		$EE\ne EO$.
		
		Case (iii).  The simplest instance of this case is when $G=K_5-e$.  Now
		$(G,h_x)$ is AT by Lemma~\ref{GeneralEulerLemma}.  In fact, that proof gives the
		stronger statement that there exists an orientation $D$ satisfying the degree
		requirements such that $EE(D)=EO(D)+1$.  In particular, $EE+EO$ is odd.
		To handle larger instances of this case, we repeatedly subdivide edges incident
		to $x$ and orient each of the resulting paths consistently, and in the direction
		of the corresponding edge in $D$.  The resulting orientation satisfies the
		degree requirements.  Further, the sum $EE+EO$ remains unchanged, and thus odd.
		Hence, still $EE\ne EO$.
	\end{proof}
	
	\begin{lem}\label{AddPathReducible}
		Let $G$ be a $T$-graph. Let $P$ be a path of $G$ where all internal
		vertices of $P$ have degree 2 in $G$ and one endvertex of $P$ has degree 2 in
		$G$.  Form $G'$ from $G$ by adding a path $P'$ (of length at least 2) joining
		the endvertices of $P$.  Now $(G', h_x)$ is AT.
	\end{lem}
	\begin{proof}
		We can assume that $G$ is not AT; otherwise, we are done by the
		\hyperlink{target:InducedSubgraph}{Subgraph Lemma}.  By symmetry, assume $P$ is a
		subpath of $P_3$. First, we get an orientation of $G$ with indegree at least 1
		for all vertices and $d^-(x) = 2$. Orient $P_1$ from $z_1$ to $x$, $P_2$ from
		$z_2$ to $x$, $P_3$ from $x$ to $z_3$, and the triangle as $\vec{z_1z_2},
		\vec{z_2z_3}$, and $\vec{z_3z_1}$.  To get an orientation of $G'$, orient the
		new path $P'$ consistently, and opposite of $P$.  Now the only directed cycle
		containing edges of $P'$ is $P'P$.  Since the Eulerian subgraph counts are
		equal for $G$, they differ by 1 for $G'$. 
	\end{proof}
	
	Now we can prove Case (2) of our \hyperlink{target:mainLemma}{Main Lemma}.  For
	reference, we restate it.

\input{pics}

	\begin{lem}
		\label{TwoConnectedClassification}
		Let $G$ be 2-connected, and choose $x\in V(G)$ with $d(x)\ge 3$. 
		Now $(G,h_x)$ is AT if and only if $G$ is not complete and $(G,h_x) \not \in \D$.
	\end{lem}
	\begin{proof}
		When $(G,h_x)\in \D$ the lemma holds by Proposition~\ref{prop:easyD}.  
		
		Now let $G$ be 2-connected, choose $x\in V(G)$ with $d(x)\ge 3$, and suppose that
		$(G,h_x)\notin \D$.  Since $G-x$ is connected, let $H'$ be a smallest connected
		subgraph of $G-x$ containing three neighbors of $x$; call these neighbors $w_1$,
		$w_2$, and $w_3$.  Consider a spanning tree $T$ of $H'$.  Since $H'$ is minimum,
		each leaf of $T$ is among $\{w_1, w_2, w_3\}$.  If $T$ is a path, then $H'$ is also a
		path.  Otherwise, $T$ is a subdivision of $K_{1,3}$.  Let $s$ be the vertex with
		$d_T(s)=3$.  If $E(G)-E(T)$ has any edge with both ends outside of $N(s)$, then
		we can delete some vertex in $N(s)$ and remain connected, contradicting the
		minimality of $H'$.  Similarly, if $N(s)$ contains at least two edges, then
		$H'-s$ still connects, $w$, $y$, and $z$.  Now let $H$ be the subgraph of $G$
		induced by $V(H')\cup\{x\}$.  Note that $H$ is either a $\theta$-graph (if $H'$
		is a tree) or a $T$-graph (if $H'$ has one extra edge in $N(s)$).  
		
		If $H$ is a $\theta$-graph, then $(G,h_x)$ is AT, by
		Lemma~\ref{ThetaReducible}.i and 
		the \hyperlink{target:InducedSubgraph}{Subgraph Lemma}.
		So assume $H$ is a $T$-graph.  Let $z_1$, $z_2$, $z_3$ be the vertices
		of degree 3 (other than $x$), and let $P_1$, $P_2$, and $P_3$ denote the paths
		from $x$ to $z_1$, $z_2$, and $z_3$; when we write $V(P_i)$, we exclude $x$ and
		$z_i$, so possibly $V(P_i)$ is empty for one or more $i\in\{1,2,3\}$.
		If any two of $P_1$, $P_2$, and $P_3$ have lengths with opposite parities, then
		we are done by Lemma~\ref{TgraphReducible}.ii; so assume not.  
		
		Now $(H,h_x)\in \D$, so we can assume that $V(G-H) \ne \emptyset$.  Choose $u
		\in V(G-H)$, and let $H_u$ be a minimal $2$-connected induced subgraph of $G$
		that contains $V(H) \cup \set{u}$.  By the \hyperlink{target:InducedSubgraph}{Subgraph Lemma} and Lemma~\ref{DegreeATClassification}, $G-x$ is a Gallai tree. 
		Thus, so is $H_u-x$; in particular, the block $B_u$ of $H_u-x$ containing $u$
		is complete or an odd cycle.  Therefore, we either have (i) $V(B_u) \cap V(H) =
		\set{z_1, z_2,z_3}$ or (ii)  $V(B_u) \cap V(H) \subseteq P_i \cup \set{z_i}$
		for some $i \in \{1,2,3\}$.
		
		Suppose (i) happens. Now $N_G(u) \cap V(H_u - x) = \set{z_1,z_2,z_3}$. If $x
		\nonadj u$, then $(G,h_x)$ is AT by the \hyperlink{target:InducedSubgraph}{Subgraph Lemma} and Lemma~\ref{T+graphReducible}.iii.  If $x \adj u$, then $x$ must have odd
		length paths to each $z_i$, by Lemma~\ref{TgraphReducible}.ii, with $u$ in the role
		of some $z_i$.  Further, $x \adj z_i$ for all $i \in \{1,2,3\}$, since
		otherwise $(G,h_x)$ is AT by the \hyperlink{target:InducedSubgraph}{Subgraph Lemma},
		Lemma~\ref{T+graphReducible}.iii, and the \hyperlink{target:SubdivideTwice}{Stretching Lemma}.  So, $H=K_4$ and $H_u=K_5$.  This implies that (ii) cannot happen for
		any vertex in $V(G-H)$, since if $V(B_u)\cap V(H)=\{z_i\}$ for some $i$, then
		$(G,h_x)$ is AT by Lemma~\ref{ThetaReducible}.i
		and the \hyperlink{target:InducedSubgraph}{Subgraph Lemma}).
		So (i) happens for every vertex in $V(G-H)$; in particular, $V(G-H)$ is joined
		to $\set{x,z_1,z_2,z_3}$.  Since $G$ is not complete, $G-x$ must contain an induced
		copy of Figure~\ref{fig:K5minus}(a); hence, $(G,h_x)$ is AT by
		Lemma~\ref{T+graphReducible}.iii and the \hyperlink{target:InducedSubgraph}{Subgraph Lemma}.
		
		Assume instead that (ii) happens for every vertex in $V(G-H)$, including $u$. 
		By symmetry, assume that $V(B_u) \cap V(H) \subseteq P_1$.  Let $z_1P_1
		= v_1v_2\cdots v_{\ell}$, where $v_{\ell}\adj x$.  First, assume that $B_u$ is
		an odd cycle of length at least $5$.  If there is $u' \in V(B_u)\setminus V(H)$
		with $u' \adj x$, then $G$ contains a $\theta$-graph and  $(G,h_x)$ is AT, by
		Lemma~\ref{ThetaReducible}.i and the \hyperlink{target:InducedSubgraph}{Subgraph Lemma}.  So, we may assume that $u' \nonadj x$ for all $u' \in V(B_u)\setminus
		V(H)$.  Now we are done by Lemma~\ref{AddPathReducible} and the
		\hyperlink{target:InducedSubgraph}{Subgraph Lemma}.
		
		So instead we assume that $B_u$ is complete. If $V(B_u) \cap V(H)=\set{v_\ell}$,
		then $G$ has an induced $\theta$-graph $J$, where $d_J(x)=d_J(v_\ell)=3$, so we are
		done by Lemma~\ref{ThetaReducible}.i and the \hyperlink{target:InducedSubgraph}{Subgraph Lemma}.  Thus, we must have $V(B_u) \cap V(H) = \set{v_{j}, v_{j+1}}$ for some
		$j \in \irange{\ell-1}$.  In particular, $B_u$ is a triangle.  If $u\nonadj x$,
		then $(G,h_x)$ is AT by the \hyperlink{target:InducedSubgraph}{Subgraph Lemma} and
		Lemma~\ref{AddPathReducible}.  So we conclude that $u\adj x$, which requires
		$j=\ell-1$, by the minimality of $H$.  Hence, $H_u$ is formed from a $T$-graph
		by adding a vertex $u$ that is adjacent to $x$ and also to the vertices of a
		$K_2$ endblock $D_u$ of $H-x$. 
		Suppose there are distinct vertices $u_1, u_2\in V(G-H)$ 
		adjacent to vertices of the same $K_2$ endblock.
		Now $G$ contains an induced copy of Figure \ref{fig:K5minus}(a), so 
		$(G,h_x)$ is AT by Lemma~\ref{T+graphReducible}.iii and the
		\hyperlink{target:InducedSubgraph}{Subgraph Lemma}.  Thus, each $K_2$ endblock has at
		most one such $u$.  
		
		Let $t$ be the number of $K_2$ endblocks in $H-x$.
		By construction, $t\le 3$; this implies that $|V(G - H)| \le t\le 3$.   
		If $t = 0$, then $G = H = K_4$, which contradicts that $G$ is not complete.  
		If $t=1$, then $G = H_u$, for the unique $u \in V(G-H)$; this is the Moser
		spindle, shown in Figure~\ref{fig:seeds}(c).  So, assume that $t \in
		\set{2,3}$.  By symmetry, assume that for each $i\in\{1,2\}$ there exists $u_i$
		such that $V(B_{u_i})\subseteq P_i\cup\{z_i\}$.  Now the subgraph induced by
		$\set{u_2}\cup V(H-P_1)$ is reducible by Lemma~\ref{AddPathReducible}. 
		So, again we are done by the \hyperlink{target:InducedSubgraph}{Subgraph Lemma}.  
	\end{proof}
	
	Taken together, Lemmas~\ref{TwoConnectedClassification}
	and~\ref{DegreeTwoVertex}, with
	Lemma~\ref{DegreeATClassification}, prove our \hyperlink{target:mainLemma}{Main Lemma}.
	However, this characterizaton requires that $G$ be 2-connected.
	Now we extend our result to the more general case, when $G$ need only be
	connected.  We use the following two definitions.  Let $G$ be a graph, $x$
	a vertex of $G$, and $B$ a block of $G$.  An \emph{$x$-lobe of
		$G$}\aside{$x$-lobe} is a maximal subgraph $A$ such that $A-x$ is connected.  A
	\emph{$B$-lobe of $G$}\aside{$B$-lobe} is a maximal subgraph $A$ such that
	$A-B$ is connected, and $A$ includes a single vertex of $B$.  
	
	\begin{thm}
		If $G$ is connected and $x \in V(G)$, then $(G, h_x)$ is not AT if and only if
		\label{thm:1connected}
	\end{thm}
	
	\begin{enumerate}
		\item[(1)] $G$ is a Gallai tree; or
		\item[(2)] $d(x) = 1$; or
		\item[(3)] $d(x) = 2$ and $G-x$ has a component that is a Gallai tree; or
		\item[(4)] $x$ is not a cutvertex, for the block $B$ of $G$ containing $x$,
		we have $(B,h_x) \in \D$, and every other block of $G$ is complete or
		an odd cycle; or
		\item[(5)]
		$x$ is a cutvertex, all but at most one $x$-lobe of $G$, say $A$, is a Gallai
		tree, and either:
		(i) $d_A(x) = 1$; or
		(ii) $d_A(x)=2$ and $A-x$ is a Gallai tree; or 
		(iii) for the block $B$ of $A$ containing $x$, 
		we have $(B,h_x)\in \D$ and all $B$-lobes of $A$ are Gallai trees. 
	\end{enumerate}
	
	\begin{proof}
		First, we check that if any of Cases (1)--(5) hold, then $(G, h_x)$ is not AT.  
		Cases (1) and (2) are immediate.  Case (3) follows from
		Lemma~\ref{DegreeTwoVertex}.
		Consider Case (4). 
		By Proposition~\ref{prop:easyD}, we know $(B,h_x)$ is not AT.
		Now $(G,h_x)$ is not AT by repeated application of Lemma~\ref{CutvertexPatch}.
		Finally, Case (5) follows from Cases (2), (3), and (4), by 
		Lemma~\ref{CutvertexPatch}.
		
		Now, for the other direction, suppose $(G, h_x)$ is not AT and none of Cases
		(1)--(5) hold.  By Lemma~\ref{DegreeTwoVertex}, and not (2) and not (3), we
		must have $d(x) \ge 3$. 
		Suppose $x$ is a cutvertex.  Now, by not (5), either (a) at least two $x$-lobes
		of $G$ are not Gallai trees or (b) $(H, h_x)$ is AT for some $x$-lobe $H$ of
		$G$.  In each case, $(G,h_x)$ is AT by Lemma~\ref{CutvertexPatch}, which is a
		contradiction.  
		
		So assume instead that $x$ is not a cutvertex.  Suppose the block $B$ of $G$
		containing $x$ is complete or $(B,h_x) \in \D$.  By not (1) and not (4),
		some $B$-lobe $H$ of $G$ is not a Gallai tree.  Since $H$ is a subgraph of
		$G-x$, and $G-x$ is connected, Lemma~\ref{DegreeATClassification} and the
		\hyperlink{target:InducedSubgraph}{Subgraph Lemma} imply that $G-x$ is degree-AT;
		hence, $(G,h_x)$ is also AT.  So, we conclude that $B$ is not complete and
		$(B,h_x)\notin \D$.
		First suppose that $d(x)=2$.  By not (3), we know that $G-x$ is not a Gallai
		tree.  Lemma~\ref{DegreeATClassification} implies that $G-x$ is degree-AT.
		So, again, the \hyperlink{target:InducedSubgraph}{Subgraph Lemma} shows that $(G,h_x)$
		is AT.  Now assume instead that $d(x)\ge 3$.  Since $(B,h_x)\notin \D$, now
		Lemma~\ref{TwoConnectedClassification} implies that $(B,h_x)$ is AT; 
		once more, the \hyperlink{target:InducedSubgraph}{Subgraph Lemma} implies that $(G,h_x)$
		is AT.
	\end{proof}
	
	\section{Choosability and Paintability}
	\label{extensions}
	As we mentioned in the introduction, Alon and Tarsi showed that if a graph $G$
	is $f$-AT, then $G$ is also $f$-choosable.  \emph{Online list coloring}, also
	called \emph{painting} is similar to list coloring, but now the list for each
	vertex is progressively revealed, as the graph is colored. 
	Schauz~\cite{schauz2010flexible} extended the Alon--Tarsi theorem, to show that
	if $G$ is $f$-AT, then $G$ is also $f$-paintable (which we define formally
	below).  In this section, we use our
	characterization of pairs $(G,h_x)$ that are not AT to prove characterizations
	of pairs $(G,h_x)$ that are not paintable and that are not choosable.  More
	precisely, a pair $(G,h_x)$ is \emph{choosable}\aaside{choosable pair}{-.3cm}
	if $G$ has a proper coloring from its lists $L$ whenever $L$ is such that
	$|L(x)|=d(x)-1$ and $|L(v)|=d(v)$ for all other $v$; otherwise $(G,h_x)$ is
	\emph{not choosable}.  A pair being \emph{paintable}\aaside{paintable
		pair}{-.3cm} is defined analogously.  We characterize all pairs $(G,h_x)$,
	where $G$ is connected and $(G,h_x)$ is not choosable (resp. not paintable).  
	In fact, we will see that these characterizations, for both choosability and
	paintability, are identical to that for pairs that are not AT.  
	
	For completeness, we include the following definition of $f$-paintable.
	Schauz~\cite{schauz2009mr} gave a more intuitive (yet equivalent) definition,
	in terms of a two player game.  We say that $G$ is \emph{$f$-paintable}
	\aside{$f$-paintable} if either (i) $G$ is empty or (ii) $f(v) \ge 1$ for all
	$v \in V(G)$ and for every $S \subseteq V(G)$ there is an independent set $I
	\subseteq S$ such that $G-I$ is $f'$-paintable where $f'(v) \DefinedAs f(v)$
	for all $v \in V(G) - S$ and $f'(v) \DefinedAs f(v) - 1$ for all $v \in S - I$.
	
	Since all pairs $(G,h_x)$ that are AT are also both paintable and choosable, it
	suffices to show that every pair $(G,h_x)$ that is not AT is also not choosable
	(here we use that if a pair is paintable, then it is also choosable).
	
	\begin{thm}
		For every connected graph $G$, the pair $(G,h_x)$ is not choosable if and only
		if $(G,h_x)$ is not AT.  Thus, the same characterization holds for pairs that
		are not paintable.
	\end{thm}
	\begin{proof}
		As noted above, every pair that is AT is also choosable and paintable.  Thus, it
		suffices to show that each pair $(G,h_x)$ in Theorem~\ref{thm:1connected} is not
		choosable.
		
		To show that Gallai trees are not degree-choosable, assign to each block $B$ a
		list of colors $L_B$ such that $|L_B|=d_B(x)$ for each $x\in V(B)$; further, for
		all distinct blocks $B_1$ and $B_2$, we require that $L_{B_1}$ and $L_{B_2}$ are
		disjoint.  For each $v\in V(G)$, let $L(v)=\cup_{B_i\ni v}L_{B_i}$.  To show
		that $G$ is not colorable from these lists, we use induction on the number of
		blocks.  Let $B$ be an endblock and $x$ a cutvertex in $B$.  Let
		$G'=G\setminus(V(G)-x)$.  Since $B$ is complete or an odd cycle, $B$ has no
		coloring from $L_B$.  Thus any coloring $\varphi$ of $G$ from $L$ does not use
		$L_B$ on $x$.  Hence, $\varphi$ gives a coloring $\varphi'$ of $G'$ from its
		lists $L'$, where $L'(x)=L(x)\setminus L_B$ and $L'(v)=L(v)$ for all $v\in
		V(G)\setminus V(B)$.  This coloring $\varphi'$ of $G'$ contradicts the induction
		hypothesis.  Thus, $G$ has no coloring from $L$.
		
		Here we use a similar approach.  Consider a pair $(G,h_x)$ that satisfies one of
		Cases (1)--(5) in Theorem~\ref{thm:1connected}.  We show that $(G,h_x)$ is not
		choosable.  Case (1) is immediate by the previous paragraph.
		Case (2) is immediate, since $|L(x)|=0$.  For Case (3), give lists to the Gallai
		tree of $G-x$ as above; now let $L(x)=\{c\}$ for some new color $c$, and add $c$
		to the list of each neighbor of $x$.  Again $G$ cannot be colored from $L$.  For
		Case (4), assign lists to $V(B)$ as in Proposition~\ref{prop:easyD} and to the
		other blocks as above.  Again, $G$ has no coloring from
		these lists.  Finally, consider Case (5).  Assign lists for all blocks outside
		of $A$ as above, and assign lists for $A$ as above in Case (2), (3), or (4).
	\end{proof}
	
	To conclude this section, we consider labelings $h_{x,y}$, where $h_{x,y}(x)=
	h_{x,y}(y)=1$ and $h_{x,y}(v)=0$ for all other $v\in V(G)$.
	We show that the set of pairs $(G,h_{x,y})$ that are not AT differs from the set
	of that are not paintable.  Further, both sets differ from
	the set of pairs that are not choosable.
	It suffices to give a pair $(G_1,h_{x,y})$ that is choosable but not paintable
	and a second pair $(G_2,h_{x,y})$ that is paintable but not AT.
	
	\begin{figure}[hbt]
		\centering
		\begin{tikzpicture}[scale = 11]
		\tikzstyle{VertexStyle} = []
		\tikzstyle{EdgeStyle} = []
		\tikzstyle{labeledStyle}=[shape = circle, minimum size = 6pt, inner sep = 1.2pt, draw]
		\tikzstyle{unlabeledStyle}=[shape = circle, minimum size = 6pt, inner sep = 1.2pt, draw, fill]
		\Vertex[style = labeledStyle, x = 0.850, y = 0.800, L = \small {$0$}]{v0}
		\Vertex[style = labeledStyle, x = 0.750, y = 0.700, L = \small {$1$}]{v1}
		\Vertex[style = labeledStyle, x = 0.950, y = 0.700, L = \small {$1$}]{v2}
		\Vertex[style = labeledStyle, x = 0.800, y = 0.550, L = \small {$0$}]{v3}
		\Vertex[style = labeledStyle, x = 0.900, y = 0.550, L = \small {$0$}]{v4}
		\Edge[label = \tiny {}, labelstyle={auto=right, fill=none}](v1)(v0)
		\Edge[label = \tiny {}, labelstyle={auto=right, fill=none}](v2)(v0)
		\Edge[label = \tiny {}, labelstyle={auto=right, fill=none}](v3)(v1)
		\Edge[label = \tiny {}, labelstyle={auto=right, fill=none}](v3)(v2)
		\Edge[label = \tiny {}, labelstyle={auto=right, fill=none}](v3)(v4)
		\Edge[label = \tiny {}, labelstyle={auto=right, fill=none}](v4)(v1)
		\Edge[label = \tiny {}, labelstyle={auto=right, fill=none}](v4)(v2)
		\begin{scope}[xshift=.12in]
		\Vertex[style = labeledStyle, x = 0.850, y = 0.800, L = \small {$0$}]{v0}
		\Vertex[style = labeledStyle, x = 0.750, y = 0.700, L = \small {$1$}]{v1}
		\Vertex[style = labeledStyle, x = 0.950, y = 0.700, L = \small {$1$}]{v2}
		\Vertex[style = labeledStyle, x = 0.800, y = 0.550, L = \small {$0$}]{v3}
		\Vertex[style = labeledStyle, x = 0.900, y = 0.550, L = \small {$0$}]{v4}
		\Edge[label = \tiny {}, labelstyle={auto=right, fill=none}](v1)(v0)
		\Edge[label = \tiny {}, labelstyle={auto=right, fill=none}](v2)(v0)
		\Edge[label = \tiny {}, labelstyle={auto=right, fill=none}](v3)(v1)
		\Edge[label = \tiny {}, labelstyle={auto=right, fill=none}](v3)(v2)
		\Edge[label = \tiny {}, labelstyle={auto=right, fill=none}](v4)(v1)
		\Edge[label = \tiny {}, labelstyle={auto=right, fill=none}](v4)(v2)
		\end{scope}
		\end{tikzpicture}
		\caption{
			The pair on the left is choosable, but not paintable.
			The pair on the right is paintable, but not AT.
			\label{fig:splits}
		}
	\end{figure}
	
	\begin{prop}
		The pair $(G_1,h_{x,y})$ on the left in Figure~\ref{fig:splits} is choosable,
		but not paintable.  The pair $(G_2,h_{x,y})$ on the right in
		Figure~\ref{fig:splits} is paintable, but not AT.
	\end{prop}
	\begin{proof}
		Let $(G_1,h_{x,y})$ denote the pair on the left, where $x$ and $y$ are the
		vertices labeled 1.  Let $(G_2,h_{x,y})$ denote the pair on the right, where
		$x$ and $y$ are the vertices labeled 1.

		We first show that $(G_1,h_{x,y})$ is choosable.  Let $L$ denote the list
		assignment.  If there exists $c\in L(x)\cap L(y)$, then use $c$ to color $x$ and
		$y$, and color the remaining vertices greedily.  So suppose there does not
		exist such a color $c$.  Let $z$ be a vertex in both triangles and note that
		there exist $c\in (L(x)\cup L(y))\setminus L(z)$.  By symmetry, assume that
		$c\in L(x)$.  Color $x$ with $c$, and color $G_1-x$ greedily, starting with the
		vertex of degree 2 and ending with $z$.
		
		We now show that $(G_1,h_{x,y})$ is not paintable. Let $S$ be the vertices of one 
		triangle.  By definition, there must be $I \subseteq S$ such that $G_1-I$ is 
		$f'$-paintable, where $f'(v) \DefinedAs f(v)$ for $v \in V(G_1) - S$ and $f'(v)
		\DefinedAs f(v) - 1$ for $v \in S - I$.  $I$ must have one vertex, $w$.  There
		are two choices for $w$; either $w$ is in two triangles or not.  If $w$ is not
		in two triangles, then $G_1 - w$ is a triangle with a pendant edge, where
		the vertices on the triangle all have list size 2, so $G_1-w$ is not paintable.
		If $w$ is one of the vertices in two triangles, then $G_1-w$ is a 4-cycle with
		list sizes alternating $1, 2, 1, 2$.  Again $G_1-I$ is not paintable (nor choosable).
		
		To see that $(G_2,h_{x,y})$ is not AT, note that any good orientation would need
		indegrees summing to at least 7, but $G_2$ has only 6 edges.  Now we show that
		$(G_2,h_{x,y})$ is paintable.  Note that $G_2$ is isomorphic to $K_{2,3}$, the complete
		biparite graph.  Call the parts $X$ and $Y$, with $|X|=2$ and $|Y|=3$.  If $S$
		includes at least two vertices of $X$ or at least two vertices of $Y$, take $I$
		to be an independent set of size at least 2.  It is easy to check that $G-I$ is
		paintable, since it induces either an independent set or a path, where each
		endvertex has more colors than neighbors.  So assume that $S$ contains at most
		one vertex from each of $X$ and $Y$.  If $S$ contains a vertex of $X$, then
		color it.  The resulting graph is paintable, since it is a claw, $K_{1,3}$,
		with at most one leaf having a single color and all other vertices having two
		colors.  Finally, suppose $S$ contains only a single vertex of $Y$. Let $I=S$. 
		The resulting graph is $C_4$, which is degree-paintable (since it is degree-AT).
	\end{proof}
	
	A graph is \emph{unstretched} \aside{unstretched} if it has no induced path
	$u_1v_1v_2u_2$ where $d(v_1)=d(v_2)=2$ (as in Corollary~\ref{ReduceP4Cor}).
	We finish with the following question.
	
	\begin{question}
		Are there only finitely many unstretched, 2-connected graphs $G$ such that 
		$(G,h_{x,y})$ is not choosable (resp.~paintable, AT)?  More generally, let
		$h_{x_1,\ldots,x_k}$ be a labeling that assigns 1 to vertices $x_1,\ldots,x_k$
		and 0 to all others.  Are there only finitely many unstretched, 2-connected
		graphs $G$ such that $(G, h_{x_1,\ldots,x_k})$ is not choosable
		(resp.~paintable, AT)?
	\end{question}
	
	
	\bibliographystyle{abbrv}
	\bibliography{GraphColoring1}
	
\end{document}

%% file: D-pics.tex
\begin{figure}[bht]
\centering
\tikzstyle{_BoldEdgeStyle} = [line width=3]
\begin{tikzpicture}[scale = 11]
\tikzstyle{VertexStyle} = []
\tikzstyle{EdgeStyle} = []
\tikzstyle{labeledStyle}=[shape = circle, minimum size = 6pt, inner sep = 1.2pt, draw]
\tikzstyle{unlabeledStyle}=[shape = circle, minimum size = 6pt, inner sep = 1.2pt, draw, fill]
\Vertex[style = labeledStyle, x = 0.650, y = 0.550, L = \small {$0$}]{v0}
\Vertex[style = labeledStyle, x = 0.850, y = 0.700, L = \small {$0$}]{v1}
\Vertex[style = labeledStyle, x = 1.050, y = 0.550, L = \small {$0$}]{v2}
\Vertex[style = labeledStyle, x = 0.850, y = 0.950, L = \small {$1$}]{v3}
\Edge[label = \tiny {}, labelstyle={auto=right, fill=none}](v1)(v0)
\Edge[label = \tiny {}, labelstyle={auto=right, fill=none}](v1)(v2)
\Edge[style = _BoldEdgeStyle, label = \tiny {}, labelstyle={auto=right, fill=none}](v1)(v3)
\Edge[style = _BoldEdgeStyle, label = \tiny {}, labelstyle={auto=right, fill=none}](v3)(v0)
\Edge[style = _BoldEdgeStyle, label = \tiny {}, labelstyle={auto=right, fill=none}](v3)(v2)
\Edge[label = \tiny {}, labelstyle={auto=right, fill=none}](v2)(v0)
\end{tikzpicture}
\begin{tikzpicture}[scale = 11]
\tikzstyle{VertexStyle} = []
\tikzstyle{EdgeStyle} = []
\tikzstyle{labeledStyle}=[shape = circle, minimum size = 6pt, inner sep = 1.2pt, draw]
\tikzstyle{unlabeledStyle}=[shape = circle, minimum size = 6pt, inner sep = 1.2pt, draw, fill]
\Vertex[style = labeledStyle, x = 0.650, y = 0.550, L = \small {$0$}]{v0}
\Vertex[style = labeledStyle, x = 0.850, y = 0.700, L = \small {$0$}]{v1}
\Vertex[style = labeledStyle, x = 1.050, y = 0.550, L = \small {$0$}]{v2}
\Vertex[style = labeledStyle, x = 0.850, y = 0.950, L = \small {$1$}]{v3}
\Vertex[style = labeledStyle, x = 0.750, y = 0.750, L = \small {$0$}]{v4}
\Vertex[style = labeledStyle, x = 0.850, y = 0.800, L = \small {$0$}]{v5}
\Vertex[style = labeledStyle, x = 0.950, y = 0.750, L = \small {$0$}]{v6}
\Edge[label = \tiny {}, labelstyle={auto=right, fill=none}](v1)(v0)
\Edge[label = \tiny {}, labelstyle={auto=right, fill=none}](v1)(v2)
\Edge[label = \tiny {}, labelstyle={auto=right, fill=none}](v2)(v0)
\Edge[style = _BoldEdgeStyle, label = \tiny {}, labelstyle={auto=right, fill=none}](v5)(v3)
\Edge[style = _BoldEdgeStyle, label = \tiny {}, labelstyle={auto=right, fill=none}](v5)(v1)
\Edge[style = _BoldEdgeStyle, label = \tiny {}, labelstyle={auto=right, fill=none}](v4)(v3)
\Edge[style = _BoldEdgeStyle, label = \tiny {}, labelstyle={auto=right, fill=none}](v4)(v0)
\Edge[style = _BoldEdgeStyle, label = \tiny {}, labelstyle={auto=right, fill=none}](v6)(v3)
\Edge[style = _BoldEdgeStyle, label = \tiny {}, labelstyle={auto=right, fill=none}](v6)(v2)
\end{tikzpicture}
\begin{tikzpicture}[scale = 11]
\tikzstyle{VertexStyle} = []
\tikzstyle{EdgeStyle} = []
\tikzstyle{labeledStyle}=[shape = circle, minimum size = 6pt, inner sep = 1.2pt, draw]
\tikzstyle{unlabeledStyle}=[shape = circle, minimum size = 6pt, inner sep = 1.2pt, draw, fill]
\Vertex[style = labeledStyle, x = 1.000, y = 0.250, L = \small {$0$}]{v0}
\Vertex[style = labeledStyle, x = 1.200, y = 0.250, L = \small {$0$}]{v1}
\Vertex[style = labeledStyle, x = 0.850, y = 0.400, L = \small {$0$}]{v2}
\Vertex[style = labeledStyle, x = 1.000, y = 0.400, L = \small {$0$}]{v3}
\Vertex[style = labeledStyle, x = 1.350, y = 0.400, L = \small {$0$}]{v4}
\Vertex[style = labeledStyle, x = 1.200, y = 0.400, L = \small {$0$}]{v5}
\Vertex[style = labeledStyle, x = 1.100, y = 0.600, L = \small {$1$}]{v6}
\Edge[label = \small {}, labelstyle={auto=right, fill=none}](v0)(v2)
\Edge[label = \small {}, labelstyle={auto=right, fill=none}](v0)(v3)
\Edge[label = \small {}, labelstyle={auto=right, fill=none}](v1)(v0)
\Edge[label = \small {}, labelstyle={auto=right, fill=none}](v2)(v3)
\Edge[label = \small {}, labelstyle={auto=right, fill=none}](v4)(v1)
\Edge[label = \small {}, labelstyle={auto=right, fill=none}](v5)(v1)
\Edge[label = \small {}, labelstyle={auto=right, fill=none}](v5)(v4)
\Edge[label = \small {}, labelstyle={auto=right, fill=none}](v6)(v2)
\Edge[label = \small {}, labelstyle={auto=right, fill=none}](v6)(v3)
\Edge[label = \small {}, labelstyle={auto=right, fill=none}](v6)(v4)
\Edge[label = \small {}, labelstyle={auto=right, fill=none}](v6)(v5)
\end{tikzpicture}

\caption{Three pairs $(G,h_x)$ that are in $\D$.  In each case $x$ is labeled 1
and all other vertices are labeled 0.  Each other pair in $\D$ can be formed
from one of these pairs by repeatedly stretching one or more bold edges.}
\label{fig:seeds}
\end{figure}

%% file: pics.tex
\begin{figure}
	\centering
\begin{tikzpicture}[scale = 8]
\tikzstyle{VertexStyle} = []
\tikzstyle{EdgeStyle} = []
\tikzstyle{labeledStyle}=[shape = circle, minimum size = 6pt, inner sep = 1.2pt, draw]
\tikzstyle{unlabeledStyle}=[shape = circle, minimum size = 6pt, inner sep = 1.2pt, draw, fill]
\Vertex[style = labeledStyle, x = 0.650, y = 0.550, L = \small {$0$}]{v0}
\Vertex[style = labeledStyle, x = 0.850, y = 0.700, L = \small {$0$}]{v1}
\Vertex[style = labeledStyle, x = 1.050, y = 0.550, L = \small {$0$}]{v2}
\Vertex[style = labeledStyle, x = 0.850, y = 0.950, L = \small {$1$}]{v3}
\Vertex[style = labeledStyle, x = 0.750, y = 0.750, L = \small {$0$}]{v4}
\Edge[label = \tiny {}, labelstyle={auto=right, fill=none}](v1)(v0)
\Edge[label = \tiny {}, labelstyle={auto=right, fill=none}](v1)(v2)
\Edge[label = \tiny {}, labelstyle={auto=right, fill=none}](v2)(v0)
\Edge[label = \tiny {}, labelstyle={auto=right, fill=none}](v4)(v0)
\Edge[label = \tiny {}, labelstyle={auto=right, fill=none}](v4)(v3)
\Edge[label = \tiny {}, labelstyle={auto=right, fill=none}](v3)(v1)
\Edge[label = \tiny {}, labelstyle={auto=right, fill=none}](v3)(v2)
\end{tikzpicture}
\label{fig:SubdividedK4}
~~~
\begin{tikzpicture}[scale = 8]
\tikzstyle{VertexStyle} = []
\tikzstyle{EdgeStyle} = []
\tikzstyle{labeledStyle}=[shape = circle, minimum size = 6pt, inner sep = 1.2pt, draw]
\tikzstyle{unlabeledStyle}=[shape = circle, minimum size = 6pt, inner sep = 1.2pt, draw, fill]
\Vertex[style = labeledStyle, x = 0.650, y = 0.550, L = \small {$0$}]{v0}
\Vertex[style = labeledStyle, x = 0.850, y = 0.700, L = \small {$0$}]{v1}
\Vertex[style = labeledStyle, x = 1.050, y = 0.550, L = \small {$0$}]{v2}
\Vertex[style = labeledStyle, x = 0.850, y = 0.950, L = \small {$1$}]{v3}
\Vertex[style = labeledStyle, x = 0.750, y = 0.750, L = \small {$0$}]{v4}
\Vertex[style = labeledStyle, x = 0.950, y = 0.750, L = \small {$0$}]{v5}
\Edge[label = \tiny {}, labelstyle={auto=right, fill=none}](v1)(v0)
\Edge[label = \tiny {}, labelstyle={auto=right, fill=none}](v1)(v2)
\Edge[label = \tiny {}, labelstyle={auto=right, fill=none}](v2)(v0)
\Edge[label = \tiny {}, labelstyle={auto=right, fill=none}](v4)(v0)
\Edge[label = \tiny {}, labelstyle={auto=right, fill=none}](v4)(v3)
\Edge[label = \tiny {}, labelstyle={auto=right, fill=none}](v5)(v2)
\Edge[label = \tiny {}, labelstyle={auto=right, fill=none}](v5)(v3)
\Edge[label = \tiny {}, labelstyle={auto=right, fill=none}](v3)(v1)
\end{tikzpicture}
	\caption{The pair $(G,h_x)$ is AT, when $G$ is formed from $K_4$ by subdividing
one or two edges incident to $x$.}
	\label{fig:TriangleRuinsPath}
\end{figure}

\begin{figure}
	\centering
\begin{tikzpicture}[scale = 8]
\tikzstyle{VertexStyle} = []
\tikzstyle{EdgeStyle} = []
\tikzstyle{labeledStyle}=[shape = circle, minimum size = 6pt, inner sep = 1.2pt, draw]
\tikzstyle{unlabeledStyle}=[shape = circle, minimum size = 6pt, inner sep = 1.2pt, draw, fill]
\Vertex[style = labeledStyle, x = 0.650, y = 0.550, L = \small {$0$}]{v0}
\Vertex[style = labeledStyle, x = 0.850, y = 0.700, L = \small {$0$}]{v1}
\Vertex[style = labeledStyle, x = 1.050, y = 0.550, L = \small {$0$}]{v2}
\Vertex[style = labeledStyle, x = 0.850, y = 0.950, L = \small {$1$}]{v3}
\Vertex[style = labeledStyle, x = 0.850, y = 0.600, L = \small {$0$}]{v4}
\Edge[label = \tiny {}, labelstyle={auto=right, fill=none}](v1)(v0)
\Edge[label = \tiny {}, labelstyle={auto=right, fill=none}](v1)(v2)
\Edge[label = \tiny {}, labelstyle={auto=right, fill=none}](v1)(v3)
\Edge[label = \tiny {}, labelstyle={auto=right, fill=none}](v2)(v0)
\Edge[label = \tiny {}, labelstyle={auto=right, fill=none}](v3)(v0)
\Edge[label = \tiny {}, labelstyle={auto=right, fill=none}](v3)(v2)
\Edge[label = \tiny {}, labelstyle={auto=right, fill=none}](v4)(v1)
\Edge[label = \tiny {}, labelstyle={auto=right, fill=none}](v4)(v0)
\Edge[label = \tiny {}, labelstyle={auto=right, fill=none}](v4)(v2)
\end{tikzpicture}
~~~
\begin{tikzpicture}[scale = 8]
\tikzstyle{VertexStyle} = []
\tikzstyle{EdgeStyle} = []
\tikzstyle{labeledStyle}=[shape = circle, minimum size = 6pt, inner sep = 1.2pt, draw]
\tikzstyle{unlabeledStyle}=[shape = circle, minimum size = 6pt, inner sep = 1.2pt, draw, fill]
\Vertex[style = labeledStyle, x = 0.650, y = 0.550, L = \small {$0$}]{v0}
\Vertex[style = labeledStyle, x = 0.850, y = 0.700, L = \small {$0$}]{v1}
\Vertex[style = labeledStyle, x = 1.050, y = 0.550, L = \small {$0$}]{v2}
\Vertex[style = labeledStyle, x = 0.850, y = 0.950, L = \small {$1$}]{v3}
\Vertex[style = labeledStyle, x = 0.750, y = 0.750, L = \small {$0$}]{v4}
\Vertex[style = labeledStyle, x = 0.850, y = 0.800, L = \small {$0$}]{v5}
\Vertex[style = labeledStyle, x = 0.950, y = 0.750, L = \small {$0$}]{v6}
\Vertex[style = labeledStyle, x = 0.850, y = 0.600, L = \small {$0$}]{v7}
\Edge[label = \tiny {}, labelstyle={auto=right, fill=none}](v1)(v0)
\Edge[label = \tiny {}, labelstyle={auto=right, fill=none}](v1)(v2)
\Edge[label = \tiny {}, labelstyle={auto=right, fill=none}](v2)(v0)
\Edge[label = \tiny {}, labelstyle={auto=right, fill=none}](v4)(v0)
\Edge[label = \tiny {}, labelstyle={auto=right, fill=none}](v4)(v3)
\Edge[label = \tiny {}, labelstyle={auto=right, fill=none}](v5)(v1)
\Edge[label = \tiny {}, labelstyle={auto=right, fill=none}](v5)(v3)
\Edge[label = \tiny {}, labelstyle={auto=right, fill=none}](v6)(v2)
\Edge[label = \tiny {}, labelstyle={auto=right, fill=none}](v6)(v3)
\Edge[label = \tiny {}, labelstyle={auto=right, fill=none}](v7)(v1)
\Edge[label = \tiny {}, labelstyle={auto=right, fill=none}](v7)(v0)
\Edge[label = \tiny {}, labelstyle={auto=right, fill=none}](v7)(v2)
\end{tikzpicture}
		\label{fig:thebigone}
	\caption{(a) The pair $(G,h_x)$ is AT, where $G=K_5-xy$. 
(b) The pair $(G,h_x)$ is AT, where $G$ is formed from $K_5-e$ by
subdividing each edge incident to $x$.}
	\label{fig:K5minus}
\end{figure}